\documentclass[11pt]{article}
\usepackage{epsf}
\usepackage{color}
\usepackage{epsfig}

\usepackage{graphicx}
\usepackage{amsmath}
\usepackage{latexsym}
\usepackage{amssymb}
\usepackage{amsfonts}

\usepackage{a4}

\newtheorem{theorem}{Theorem}

\newtheorem{proposition}[theorem]{Proposition}
\newenvironment{proof}[1][Proof]{\textbf{#1.} }
     {\    \rule{0.5em}{0.5em}}

\newcommand{\e}{\ensuremath{\mathrm{e}}}
\newcommand{\tr}{\mathrm{tr}}

\begin{document}

\title{Error analysis of splitting methods for the time dependent
Schr\"odinger equation}

\author{Sergio Blanes$^{1}$\thanks{Email: \texttt{serblaza@imm.upv.es}}
   \and
  Fernando Casas$^{2}$\thanks{Email: \texttt{Fernando.Casas@uji.es}}
   \and
Ander Murua$^{3}$\thanks{Email: \texttt{Ander.Murua@ehu.es}}
   }

\maketitle

\begin{abstract}

A typical procedure to integrate numerically the time dependent Schr\"o\-din\-ger equation involves two
stages. In the first one carries out a space discretization of the continuous problem. This results 
in the linear system of differential equations 
$i du/dt = H u$, where $H$ is a real symmetric matrix, whose solution with
initial value $u(0) = u_0 \in \mathbb{C}^N$ is given by $u(t) = \e^{-i t H} u_0$. Usually, this
exponential matrix is expensive to evaluate, so that time stepping methods to construct approximations to
$u$ from time $t_n$ to $t_{n+1}$ are considered in the second phase of the procedure. Among them,
schemes involving multiplications of the matrix $H$ with vectors, such as Lanczos and Chebyshev methods, are particularly efficient.

In this work we consider a particular class of splitting methods which also involves only products $Hu$. We
carry out an error analysis of these integrators and propose a 
strategy which allows us to construct 
 different splitting symplectic methods of different order (even of order zero) possessing a large stability interval that can be adapted to different space regularity conditions and different accuracy ranges of the spatial discretization.
The validity of the 
procedure and the performance of the resulting schemes are illustrated on several numerical examples.

\vspace*{0.6cm}

\begin{description}
 \item $^1$Instituto de Matem\'atica Multidisciplinar,
  Universidad Polit\'ecnica de Valencia, E-46022  Valencia, Spain.
 \item $^2$Institut de Matem\`atiques i Aplicacions de Castell\'o and
  Departament de Matem\`atiques, Universitat Jaume I,
  E-12071 Castell\'on, Spain.
 \item $^3$Konputazio Zientziak eta A.A. saila, Informatika
Fakultatea, EHU/UPV, Donostia/San Sebasti\'an, Spain.
\end{description}

\end{abstract}

\section{Introduction}

To describe and understand the dynamics and evolution of many basic atomic and molecular
phenomena, their time dependent quantum mechanical treatment is essential. Thus,
for instance, in molecular dynamics, the construction of models and simulations of molecular encounters
can benefit a good deal from time dependent computations. The same applies to scattering
processes such as atom-diatom collisions and triatomic photo-dissociation and, in general, to
quantum mechanical phenomena where there is an initial state that under the influence of a given
potential evolves through time to achieve a final asymptotic state (e.g., chemical reactions,
unimolecular breakdown, desorption, etc.). This requires, of course, to solve the time dependent
Schr\"odinger equation ($\hbar=1$)
\begin{equation}   \label{Schr0}
  i \frac{\partial}{\partial t} \psi (x,t)  = \hat{H} \psi (x,t),
\end{equation}
where $\hat{H}$ is the Hamiltonian operator,
$\psi:\mathbb{R}^d\times\mathbb{R} \longrightarrow \mathbb{C}$ is
the wave function representing the state of the system and the initial state is $\psi(x,0)=\psi_0(x)$. Usually
\begin{equation}   \label{Ham1}
  \hat{H} = \hat{T}(\hat{P}) + \hat{V}(\hat{X}) \equiv \frac{1}{2\mu} \hat{P}^2 + \hat{V}(\hat{X})
\end{equation}
and the operators
$\hat{X}$, $\hat{P}$ are defined by their actions on $\psi(x,t)$ as
\[
  \hat{X} \psi (x,t) = x\, \psi (x,t), \quad\quad
  \hat{P} \ \psi (x,t) = -i \, \nabla \psi (x,t).
\]
The solution of (\ref{Schr0}) provides all dynamical information on the physical system at any time. It
can be expressed as
\begin{equation}  \label{evol1}
  \psi(x,t) = \hat{U}(t) \psi_0(x),
\end{equation}
where $\hat{U}$ represents the evolution operator, which is linear and satisfies the equation
$i \, d\hat{U}(t)/dt =
\hat{H} \hat{U}(t)$ with $\hat{U}(0) = I$. Since
the  Hamiltonian is explicitly time independent, the evolution operator is given formally by
\begin{equation}  \label{evol2}
   \hat{U}(t) = \e^{-i t \hat{H}}.
\end{equation}
In practice, however, the Schr\"odinger equation has to be solved numerically, and the procedure involves
basically two steps. The first one consists in considering a faithful discrete spatial representation of the  initial wave function $\psi_0(x)$ and the operator $\hat{H}$
on an appropriately constructed grid. Once this spatial discretization is built,
the initial wave function is propagated in time until the end of the dynamical event. It is on the second stage
of this process where we will concentrate our analysis. 

As a result of the discretization of  eq. (\ref{Schr0}) in space, one is left with a linear equation 
$i du/dt = H u$,
with a Hermitian matrix $H$ of large dimension and large norm. Since evaluating exactly 
the exponential $\exp(-i t H)$ is computationally expensive,
approximation methods requiring only matrix-vector products with $H$ are particularly
appropriate \cite{lubich08fqt}. Among them, the class of splitting symplectic methods has received
considerable attention in the literature 
\cite{gray96sit,mclachlan97osp,zhu96nmw,blanes06sso,blanes08otl}.
 In this case $\exp(-i t H)$ is approximated by a composition of symplectic matrices. 
While it has been shown that stable high order  methods belonging to this family do exist, 
such high degree of accuracy may be disproportionate in comparison
with the error involved in the spatial discretization, and also inappropriate particularly when the problem at hand
involves non-smooth solutions, as high order methods make small phase errors in the low frequencies but much larger errors in the high frequencies. 
An error analysis of this family of
integrators, in particular,  could help one to design different efficient time integrators adapted to different accuracy requirements and spacial regularity situations. 

The analysis
carried out in the present paper could be considered a step forward in this direction. 
 We present a strategy which allows us to construct 
 different splitting symplectic methods of different order and large stability interval  (with a large number of stages) that can be adapted to different space regularity conditions and different accuracy ranges of the spatial discretization. When this
regularity degree is low, sometimes the best option is provided by methods of order zero.

Since the splitting methods we analyze here only involve products of the matrix $H$ with vectors, they
belong to the same class of integrators as the Chebyshev and Lanczos methods, in the sense that
all of them approximate $\exp(-i t H)u_0$ by linear combinations of terms of the form $H^j u_0$
($j \ge 1$).

The plan of the paper is as follows. In section \ref{sec.2} we review first the Fourier collocation approach
carrying out the spatial discretization of the Schr\"odinger equation, and
then we turn our attention to the time discretization errors of
symplectic splitting methods. 
The bulk of the paper is contained in section \ref{sec.3}. There we carry out a theoretical analysis
of symplectic splitting methods and obtain some estimates on the time discretization error. These
estimates in turn allows us to build different classes of splitting schemes in section \ref{sec.4}, which are then illustrated 
in section  \ref{sec.5} on several
numerical examples exhibiting different degrees of regularity. Here we also include, for comparison,
results achieved by the Lanczos and Chebyshev methods.  Finally, section \ref{sec.6} contains some conclusions.

\section{Space and time discretization}
\label{sec.2}

Among many possible ways to discretize the Schr\"odinger equation in space, collocation spectral methods
possess several attractive features: they allow a relatively small grid size for representing the wave
function, are simple to implement and provide an extremely high order of accuracy if the solution of the
problem is sufficiently smooth \cite{fornberg98apg,gottlieb77nao}. In fact, spectral methods are superior to local methods (such as finite difference schemes) not only when very high spatial resolution is required, but also when long time integration
is carried out, since the resulting spatial discretization does not cause a deterioration of the phase error
as the integration in time goes on \cite{hesthaven07smf}.

To simplify the treatment, we will limit ourselves to the one-dimensional case  and
assume that the wave function is negligible outside an interval $[\alpha, \beta]$. In such a situation
one may reformulate the problem on the finite interval with periodic boundary conditions. After rescaling, one may assume without loss of generality that the space interval is $[0, 2 \pi]$, and therefore
\begin{equation}  \label{Schr1}
 i \frac{\partial}{\partial t} \psi (x,t) =
     -\frac{1}{2\mu} \frac{\partial^2 \psi}{\partial x^2}(x,t) +
     V(x)  \psi (x,t), \qquad 0 \le x < 2 \pi
\end{equation}
with $\psi(0,t) = \psi(2\pi,t)$ for all $t$. In the Fourier-collocation (or pseudospectral) approach, one
intends to construct approximations based on the equidistant interpolation grid
\[
    x_j = \frac{2 \pi}{N} j, \qquad j= 0, \ldots, N-1
\]
where $N$ is even (although the formalism can also be adapted to an odd number of points). Then
one seeks  a solution of the form \cite{lubich08fqt}
\begin{equation}   \label{pseudo1}
   \psi_N(x,t) = \sum_{|n| \le N/2} c_n(t) \e^{i n x}, \qquad x \in [0, 2 \pi)
\end{equation}
where the coefficients $c_n(t)$ are related to the grid values $\psi_N(x_j,t)$
through a discrete
Fourier transform of length $N$, $\mathcal{F}_N$ \cite{trefethen00smi}. Its
computation  can be accomplished by the \emph{Fast Fourier Transform} (FFT) algorithm
with $\mathcal{O}(N \log N)$ floating point operations.

In the collocation approach, the grid values $\psi_N(x_j,t)$ are determined by
requiring that the approximation (\ref{pseudo1}) satisfies the Schr\"odinger equation precisely
at the grid points $x_j$ \cite{lubich08fqt}.
This yields a system of $N$ ordinary differential equations to determine the $N$ point values
$\psi_N(x_j,t)$:
\begin{equation}   \label{pseudo8}
   i \frac{du}{dt} = \mathcal{F}_N^{-1} D_N \mathcal{F}_N \, u + V_N u \equiv H u,  \qquad
    u = (u_0, u_1, \ldots, u_{N-1}),
\end{equation}
where
\begin{equation}  \label{pseudo8b}
    D_N = \frac{1}{2\mu} \mbox{diag}(n^2), \qquad V_N = \mbox{diag}( V(x_j))
\end{equation}
for $n = -N/2, \ldots, N/2 - 1$ and $j = 0, \ldots, N-1$. Observe that the matrices on the right-hand side
of (\ref{pseudo8}) are Hermitian.

An important qualitative feature of this space discretization procedure is that it replaces the original Hilbert
space $\mathcal{L}^2(0,2\pi)$ defined by the quantum mechanical problem by a discrete one in which the action of operators are
approximated by $N \times N$ (Hermitian) matrices obeying the same quantum mechanical commutation
relations \cite{kosloff83afm}. From a quantitative point of view, if the function $\psi$ is sufficiently smooth and periodic, then the
coefficients $c_n$ exhibit a rapid decay (in some cases, faster than algebraically in $n^{-1}$,
uniformly in $N$), so that  typically the value of $N$ in the expansion (\ref{pseudo1}) needs not to be
very large to represent accurately the solution. Specifically, in \cite{lubich08fqt} the following result is
proved.
\begin{theorem}
\label{th:lubich}
Suppose that the exact solution $\psi(x,t)$ of (\ref{Schr1}) is such that, for some $s \ge 1$,
$\partial_x^{s+2} \psi(\cdot,t) \in \mathcal{L}^2(0,2\pi)$ for every $t \ge 0$. Then the error due to the approximation $\psi_N(x,t)$ defined by (\ref{pseudo1}) in the collocation approach is bounded by
\[
   \|\psi_N(\cdot,t) - \psi(\cdot,t)\| \le C \, N^{-s} (1+t) \max_{0 \le t' \le t} \, \left\|\partial_x^{s+2} \psi(\cdot,t') \right\|,
\]
where $C$ depends only on $s$.
\end{theorem}

When the problem is not periodic,  the use of a truncated Fourier series introduces 
errors in the computation. In that case several techniques have been proposed
to minimize its effects (see \cite{balakrishnan97tdq,boyd01caf} and references therein).

The previous treatment can be generalized to several spatial dimensions, still exploiting all the
one-dimensional features, by taking tensor products of one-dimensional expansions. 
The resulting functions are then defined on the Cartesian product of intervals \cite{canuto06smf,lubich08fqt}.

We can then conclude that after the previous space discretization has been applied to eq. (\ref{Schr1}),
one ends up with a linear system of ODEs of the form
\begin{equation} \label{td.1}
  i \frac{d }{dt} u(t) = H u(t), \qquad u(0)=u_0 \in \mathbb{C}^N,
 \end{equation}
where $H$ is a real symmetric matrix. This is the starting point for carrying out an integration in time. Although
a collocation approach has been applied here, in fact
any space discretization scheme leading to an equation of the form (\ref{td.1}) fits in our
subsequent analysis.

  The spatial
 discretization chosen has of course a direct consequence on the time propagation of the (discrete)
 wave function $u(t)$, since the matrix $H$ representing the Hamiltonian
 has a discrete spectrum which depends on the scheme. This discrete representation, in addition,
 restricts the energy range of the problem and therefore imposes an upper bound to the high frequency
 components represented in the propagation \cite{leforestier91aco}.

 The exact solution of eq. (\ref{td.1}) is given by
 \begin{equation}   \label{td.2}
    u(t) = \e^{-i t H} \, u_0,
\end{equation}
but to compute the exponential of the $N \times N$ complex and full matrix $-i t H$ (typically also of large
norm) by diagonalizing the matrix $H$ can be prohibitively expensive for large values of $N$. In practice, thus, one turns to
time stepping methods advancing the approximate solution from time $t_n$ to $t_{n+1} = t_n + \Delta t$,
so that the aim is to construct an approximation 
$u_{n+1} \approx u(t_{n+1}) = \e^{-i \Delta t \, H} u(t_n)$ as a map
$u_{n+1} = \phi_{\Delta t} u_n$. 

Among them, exponential splitting schemes have been widely used when the 
Hamiltonian has the form given by 
(\ref{Ham1}) \cite{feit82sot,leforestier91aco,lubich08fqt}. 
In that case, equation (\ref{td.1}) reads
\begin{equation}   \label{td.3}
   i \, \dot{u} = (T + V) u, \qquad u(0) = u_0,
\end{equation}
where $V$ is a diagonal matrix associated with $\hat{V}$ and $T$ is
related to the kinetic energy $\hat{T}$. It turns out that  the solutions $\e^{-i t T} u_0$  and $\e^{-i t V} u_0$
of equations
    $i \dot{u} = T u$  and   $i \dot{u} = V u$, respectively, can be easily determined
  \cite{lubich08fqt}, so that one may
 consider compositions of the form 
 \begin{equation}   \label{td.9}
  \e^{-i b_m \tau V} \, \e^{-i a_m \tau T} \cdots \e^{-i b_1 \tau V} \, \e^{-i a_1 \tau T},
\end{equation}
where $\tau \equiv \Delta t$. In (\ref{td.9}) the number of exponentials $m$ (and therefore the number of
coefficients $\{a_i,b_i\}_{i=1}^m$) has to be sufficiently large
to solve all the equations required to achieve order $r$ (the so called order conditions). 

Splitting methods of this class have several structure-preserving properties. They are
unitary, so that the norm of $u$ is
preserved along the integration, and 
time-reversible when the composition (\ref{td.9}) is symmetric. Error estimates of such methods applied to the Schrodinger equation \cite{jahnke00ebf,neuhauser09otc,thalhammer08hoe}  seem to suggest that, while they are indeed very efficient for high spatial regularity, they may not be very appropriate under conditions
of limited regularity.

Here we will concentrate on another class of splitting methods that have been considered in the literature~\cite{gray96sit,mclachlan97osp,zhu96nmw,blanes06sso,blanes08otl}.
Notice that the corresponding $H$ in eq. (\ref{pseudo8})  is a real symmetric matrix, and thus
$\e^{-it H}$ is not only unitary, but also symplectic with canonical coordinates
$q = \mbox{Re}(u)$ and momenta $p = \mbox{Im}(u)$. In consequence,
equation (\ref{td.1}) is equivalent to \cite{gray96sit,gray94chs}
\begin{equation}  \label{td.10}
  \dot{q} = H p, \qquad \dot{p} = - H q.
\end{equation}
Alternatively, one may write
\begin{equation}  \label{td.11}
 \frac{d}{dt} \left(
 \begin{array}{c}
  q \\
  p  \end{array} \right) =  \left(
 \begin{array}{ccc}
  0   & \, & H  \\
  -H & \, &  0  \end{array} \right)   \left(
 \begin{array}{c}
   q \\
   p  \end{array} \right) \equiv ( A+ B) \left(
 \begin{array}{c}
  q \\
  p  \end{array} \right),
\end{equation}
with the $2N \times 2N$ matrices $A$ and $B$
given by
\[
  A = \left(  \begin{array}{ccc}
               0   & \, &  H  \\
               0    & \, &  0  \end{array} \right),
              \qquad\qquad
  B = \left(  \begin{array}{ccc}
               0   & \, & 0  \\
              -H    & \, & 0  \end{array} \right).
\]
The solution operator corresponding to (\ref{td.11}) can be written in terms of the rotation matrix
\begin{equation}  \label{O(y)}
   O(y) = \left(
 \begin{array}{rcr}
   \cos(y)   &  & \sin(y)  \\
  -\sin(y)  &   & \cos(y)  \end{array}
  \right)
\end{equation}
as $O(t \, H)$, which is an orthogonal and symplectic $2N \times 2N$ matrix. 
Computing $O(t\, H)$ exactly (by diagonalizing the matrix $H$) is, as mentioned before 
for its complex representation $\e^{-i\, t \, H}$,
computationally very expensive, so that one typically  splits
the whole time interval into subintervals of length $\tau \equiv \Delta t$ and then
approximate $O(\tau H)$ acting on the initial condition at each step. Since
\[
  \e^{\tau a_k A} = \left(
 \begin{array}{ccc}
    I   & \,  & a_k \tau H  \\
    0  & \,  &  I   \end{array} \right), \qquad \qquad
   \e^{\tau b_k B} = \left(
 \begin{array}{ccc}
   I   & \, & 0 \\
  -b_k \tau H  &  \, & I  \end{array} \right)
\]
it makes sense to apply splitting methods of the form 
\begin{equation} \label{compos}
  u_{n+1} = \e^{\tau b_{m} B} \, \e^{\tau a_m A} \ \cdots \
  \e^{\tau b_1 B} \, \e^{\tau a_1 A } \, u_n.
\end{equation}
Observe that the evaluation of the exponentials of $A$ and $B$ requires only  computing the products $H p$ and
$H q$, and this can be done very efficiently with the FFT algorithm.

Several methods with different orders have been constructed along
these lines \cite{gray96sit,liu05oos,zhu96nmw}. In particular,  the schemes presented in \cite{gray96sit}
use only $m=r$
exponentials $\e^{\tau a_{i} A}$ and $\e^{\tau b_{i} B}$ to achieve order
$r$ for $r=4,6,8,10$ and $12$. Furthermore, when the idea of processing is taken into account,
it is possible to design families of symplectic splitting methods with large stability intervals and
a high degree of accuracy \cite{blanes06sso,blanes08otl}. They have the general structure
$P(\tau H) K(\tau H) P^{-1}(\tau H)$,
where $K$ (the kernel) is built as a composition (\ref{compos}) and $P$ (the processor) is taken
as a polynomial. 

Although these methods are neither unitary nor unconditionally stable, they are symplectic and
conjugate to unitary schemes. In consequence, neither the average error in energy
nor the norm of the solution increase with time. In other words, 
the quantities $\|u\|^2 = u^T \bar{u}/N$ and $u^T H \bar{u}/(2N)$ are both approximately preserved
along the evolution, since the committed error is (as shown in Subsection~\ref{ssec.3.1} below) only local and does not propagate with time. 
The mechanism that takes place here is analogous to the propagation of
the error in energy for symplectic integrators in classical
mechanics \cite{hairer06gni}. In addition, the families of splitting methods considered here are designed
to have large stability intervals and can be applied when no particular structure is required for
the Hamiltonian matrix $H$. Furthermore, they can also be used in more general problems of the
form $\dot{q} = M_1 p$, $\dot{p} = - M_2 q$, resulting, in particular, from the space discretization of Maxwell
equations \cite{blanes08otl}.

\section{Analysis of symplectic splitting methods for time discretization}
\label{sec.3}

In this section we proceed to characterize the family of splitting symplectic methods (\ref{compos}),
paying special attention to their stability properties. By interpreting the numerical solution as the exact
solution corresponding to a modified differential equation, it is possible to prove that the norm and
energy of the original system are approximately preserved along evolution. We also
provide rigorous estimates of the time discretization error that are uniformly valid as both the space
and time discretizations get finer and finer. The analysis allows us to construct new methods with large stability
domains such that the error introduced is comparable to the error coming from the space discretization.

\subsection{Theoretical analysis}
\label{ssec.3.1}

It is clear  that 
the problem of finding appropriate compositions  of the form (\ref{compos}) for equation
(\ref{td.11}) is equivalent to getting coefficients
$a_{i}$, $b_{i}$ in the matrix 
\begin{equation}
  K(\tau H)  =   
   \left(
 \begin{array}{cr}
  I & 0 \\
   -b_m \tau H & \  I  \end{array} \right)
  \left(
 \begin{array}{cc}
  I & a_m \tau H \\
  0  & \ \ I  \end{array} \right)
    \ \cdots \
   \left(
 \begin{array}{cr}
  I & 0 \\
   -b_1 \tau H & \  I  \end{array} \right)
  \left(
  \begin{array}{cc}
  I & a_1 \tau H \\
   0 & \ \ I  \end{array} \right)
  \label{prod-gen2p2}
\end{equation}
such that $K(\tau H)$
approximates the solution $O(\tau H)$, where $O(y)$ denotes the rotation matrix (\ref{O(y)}).
The matrix $K(\tau H)$ that propagates the numerical solution of the splitting method (\ref{prod-gen2p2})  can be written as
\begin{equation}\label{K(x)}
  K(\tau H) =  \left(
 \begin{array}{cc}
  K_{1}(\tau H) &  K_{2}(\tau H) \\
  K_{3}(\tau H) &  K_{4}(\tau H)
 \end{array} \right),
\end{equation}
where the entries $K_1(y)$ and $K_4(y)$  (respectively, $K_2(y)$ and $K_3(y)$) are even 
(repect., odd) polynomials in $y \in \mathbb{R}$, and $\det K(y) = K_1(y) K_4(y) - K_2(y) K_3(y) \equiv 1$. 
It is worth stressing here that by diagonalizing the matrix $H$ with an appropriate linear change of variables,
one may transform the system into $N$ uncoupled harmonic oscillators
with frequencies $\omega_1,\ldots,\omega_{N}$. Although in practice one wants to avoid diagonalizing $H$, 
numerically solving system (\ref{td.10}) by a splitting method is mathematically equivalent to 
applying the splitting method to each of such one-dimensional harmonic oscillators (and then 
rewriting the result in the original variables). Clearly, the numerical solution of each 
individual harmonic oscillator is propagated by 
the $2\times 2$ matrix $K(y)$ with polynomial entries $K_j(y)$ ($j=1,2,3,4$) for $y=\tau \omega_j$. 
We will refer to $K(y)$ in the sequel as the propagation matrix, although other denominations have 
also been used 
\cite{blanes08otl,mclachlan97osp}.

Moreover, for a given $K(y)$ with polynomial entries, an algorithm has been proposed to  
factorize $K(y)$ as (\ref{prod-gen2p2}) and determine uniquely the coefficients $a_i$, $b_i$ of the
splitting method \cite[Proposition 2.3]{blanes08otl}. Thus, any splitting method  is uniquely determined by its propagation matrix $K(y)$. For this reason, in the analysis that follows we will be only 
concerned with such matrices $K(y)$.

 When applying splitting methods to the system (\ref{td.10}) with time step size $\tau$, 
 the numerical solution is propagated by $(K(\tau H))^n$ as an approximation to 
 $O(\tau H)^n=O(n\tau H)$, which is bounded (with $L^2$ norm equal to 1) independently of $n$. 
 It then makes sense requiring that
$(K(\tau H))^n$ be also bounded independently of $n \ge 1$. This clearly holds if for 
each eigenvalue $\omega_j$ of $H$, the corresponding $2 \times 2$ matrix $K(y)$  
with $y = \tau \omega_j$ is stable, i.e., if $\|(K(\tau \omega_j))^n\| \le C$ for some constant $C > 0$. 

In our analysis, use will be made of the stability polynomial, defined for a given $K(y)$ by
\begin{equation}    \label{stab-pol}
   p(y) = \frac{1}{2} \tr \, K(y) = \frac{1}{2} ( K_1(y) + K_4(y)).
\end{equation}
The following proposition, whose proof can be found in \cite{blanes08otl}, provides a characterization of the
stability of $K(y)$.

\begin{proposition}
\label{p:1} Let $K(y)$ be a $2 \times 2$  matrix  with $\det K(y)
= 1$,  and $p(y)=\frac{1}{2} \tr\, K(y)$, with $y \in \mathbb{R}$.  
Then, the following statements are
equivalent:
\begin{itemize}
  \item[(a)] The matrix $K(y)$ is stable.
  \item[(b)]  The matrix $K(y)$ is diagonalizable with eigenvalues of modulus one.
  \item[(c)]  $|p(y)|\leq 1$, and there exists a real matrix $Q(y)$ such that 
\begin{eqnarray} \label{Ophi(x)}
 Q(y)^{-1} K(y) Q(y) = O(\phi(y)), 
\end{eqnarray}
where $O(y)$ is the rotation matrix (\ref{O(y)}) and 
$\phi(y) = \mathrm{arccos} \, p(y) \in \mathbb{R}$.
\end{itemize}
\end{proposition}

We define the stability threshold $y_*$ as the largest non negative real number such that $K(y)$
is stable for all $y \in (-y_*,y_*)$. Thus, $(K(\tau H))^n$ will be bounded independently of $n \ge 1$ 
if all the eigenvalues of $\tau H$ lie on the stability interval $(-y_*,y_*)$, that is, if $\tau \, \rho(H)<y_*$, 
where $\rho(H)$ is the spectral radius of the matrix $H$. For instance, if a Fourier-collocation approach 
based on $N$ nodes is applied to discretize (\ref{Schr1})
in space, the spectral radius is of size $\rho(H) = \mathcal{O}(N^2)$ 
(cf. eqs. (\ref{pseudo8})-(\ref{pseudo8b})), 
which shows that $\tau$ must decrease proportionally to $N^{-2}$ as the number of 
nodes $N$ increases.

The stability threshold
$y_*$ depends on the coefficients $\{a_i, b_i\}$ of the method (\ref{compos}) and verifies $y_* \le 2m$, since
$2m$ is the optimal value for the stability  threshold achieved by the concatenation of $m$
steps of length $\tau/m$ of the leapfrog scheme \cite{chawla81iop}.

The stability of the matrix $K(y)$ of a splitting method for a given $y \in \mathbb{R}$ can 
alternatively be characterized as follows.
\begin{proposition}
\label{p:2} The matrix $K(y)$ is stable for a given $y \in \mathbb{R}$ if and only if there exist 
real quantities $\phi(y),\epsilon(y),\gamma(y)$, with $\gamma(y) \ne 0$, such that
\begin{eqnarray}  \label{alt-stab}
     K(y)=
\left(
 \begin{array}{cc}
    \cos (\phi(y)) + \epsilon(y) \sin( \phi(y)) &   \gamma(y) \sin( \phi(y))  \\
\displaystyle -\frac{1+\epsilon(y)^2}{\gamma(y)} \sin ( \phi(y)) &  \cos( \phi(y)) - \epsilon(y) \sin(\phi(y))
 \end{array} \right).
\end{eqnarray}
\end{proposition}
\begin{proof}
If $K(y)$ is of the form (\ref{alt-stab})  then, obviously, $\tr K(y) = 2 \cos(\phi(y))$ and  thus
$\cos(\phi(y)) = p(y)$. Moreover, it is straightforward to check that  (\ref{Ophi(x)}) holds with 
\begin{eqnarray}
\label{Q(y)}
  Q(y) = 
\left(
 \begin{array}{cc}
\gamma(y)^{1/2} &  0  \\
 -\epsilon(y)\, \gamma(y)^{-1/2}  &  \gamma(y)^{-1/2}
 \end{array} \right),
\end{eqnarray}
so that $K(y)$ is stable in that case. 

Let us assume now that $K(y)$ is stable, so that from the third characterization 
 given in Proposition~\ref{p:1},  $\phi(y)=\arccos(p(y)) \in \mathbb{R}$, where $p(y)$ is 
 the stability polynomial. We now consider two cases:
\begin{itemize}
\item   $p(y)=1$ (resp. $p(y)=-1$), so that $K(y)$ (resp. $-K(y)$) is similar to the identity matrix, which implies that $K(y)$  (resp. $-K(y)$) is also the identity matrix. In that case, (\ref{alt-stab}) holds with $\epsilon(y)=0$ and $\gamma(y)=1$.
\item If $p(y)^2 \neq 1$, then $\sin(\phi(y)) \neq 0$, and we set
  \begin{eqnarray*}
    \epsilon(y) = \frac{K_1(y)-K_4(y)}{2\sin(\phi(y))}, \qquad
    \gamma(y) = \frac{K_2(y)}{\sin(\phi(y))}. 
  \end{eqnarray*}
Since $\det(K(y))=1$, one has
\begin{eqnarray*}
  -K_2(y)K_3(y) = 1-K_1(y)K_4(y)
= (1+\epsilon(y)^2)\sin(\phi(y))^2,
\end{eqnarray*}
which implies that $\gamma(y)\neq 0$ and 
\begin{eqnarray*}
  K_3(y)= -\frac{1+\epsilon(y)^2}{\gamma(y)}\sin(\phi(y)).
\end{eqnarray*}
\end{itemize}
\end{proof}

Notice that, for a given splitting method with a non-empty stability interval $(-y_*,y_*)$,  Proposition~\ref{p:2} determines two odd functions $\phi(y)$ and $\epsilon(y)$ and an even function $\gamma(y)$ defined for $y \in (-y_*,y_*)$ which characterize the accuracy of the method when applied with step size $\tau$ to a harmonic oscillator of frequency $\omega$, with $y=\tau \omega$. An accurate approximation will be obtained if $|\phi(y)-y|$, $|\gamma(y)-1|$, and $|\epsilon(y)|$ are all small quantities.   
In particular,  if the splitting method is of order $r$, then
\[
   \phi(y) = y + \mathcal{O}(y^{r+1}), \qquad\quad   \epsilon(y) = \mathcal{O}(y^{r}), \qquad\quad  \gamma(y) = 1 + \mathcal{O}(y^{r})
\]      
as $y \rightarrow 0$. For instance, for the simple first order splitting 
$\e^{\tau A} \e^{\tau B}$ one has
\[
K(y) = 
 \left(
 \begin{array}{cr}
  1 & y \\
   0 &  1  \end{array} \right)
 \left(
 \begin{array}{cr}
  1 & 0 \\
   - y &  1  \end{array} \right) =
 \left(
 \begin{array}{cr}
  1-y^2 & y \\
   - y &  1  \end{array} \right),
\]
and one can easily check that 
\begin{eqnarray*}
  \phi(y)&=& \arccos(1-y^2/2) = 2 \arcsin(y/2) = y + \mathcal{O}(y^3), \\
\epsilon(y) &=& \frac{-y}{\sqrt{4-y^2}} = \mathcal{O}(y), \\
\gamma(y) &=& \frac{2}{\sqrt{4-y^2}} = 1 + \mathcal{O}(y^2).
\end{eqnarray*}

It is worth stressing that (\ref{alt-stab}) implies that (\ref{Ophi(x)}) holds with 
$Q(y)$ given by (\ref{Q(y)}). This feature, in particular, allows
us to interpret the numerical result obtained by a
splitting method of the form (\ref{compos}) applied to (\ref{td.11}) in terms of the 
exact solution corresponding to a 
modified differential equation. Specifically,  assume that 
$q_n+ i\, p_n =u_n \approx u(t_n) = \exp(-i\, t_n  H) u_0$ is obtained (for $t_n = n \tau$, $n\geq 1$) as
\[
     \left(
 \begin{array}{c}
q_n\\ p_n 
 \end{array} \right) =  K(\tau H)^n 
 \left( \begin{array}{c}
q_0\\ p_0
 \end{array} \right).
\]
If $\tau \, \rho(H) <  y_*$, 
then it holds that
\begin{eqnarray*}
   \tilde{u}_n \equiv  \big( \gamma(\tau  H)^{-1/2} + i  \, \epsilon(\tau H) \gamma(\tau  H)^{-1/2} \big) \, q_n+ i\, \gamma(\tau  H)^{1/2} \, p_n
\end{eqnarray*}
trivially verifies $\tilde{u}_n = \exp(-i n \phi(\tau H)) \tilde{u}_0$. In other words, $\tilde{u}_n$ \textit{is}
the exact solution at $t_n = n \tau$ of the initial value problem 
\begin{eqnarray}
\label{eq:mode}
i \frac{d}{dt}  \tilde{u} =  \tilde{H} \, \tilde{u}, \qquad 
 \tilde{u}(0) =  \tilde{u}_0,
\end{eqnarray}
where $\tilde H = \frac{1}{\tau} \phi(\tau H) \approx H$.  With this backward error
analysis interpretation at hand, it readily follows the preservation of both the discrete $L^2$ norm  of 
\begin{eqnarray*}
  \tilde{u} = \big( \gamma(\tau  H)^{-1/2} + i  \, \epsilon(\tau H) \gamma(\tau  H)^{-1/2} \big) \, q +
   i\, \gamma(\tau  H)^{1/2} \, p
\end{eqnarray*}
 and the energy
 corresponding to (\ref{eq:mode}). This implies that the discrete $L^2$ norm of $u=q+ i\, p$ and the energy of the original system will be approximately preserved (that is, their variation will be uniformly bounded for all times $t_n$).

\subsection{Error estimates}

Our goal now is to obtain meaningful estimates of the time dis\-cre\-ti\-za\-tion error that are uniformly valid as $N\to \infty$ and $\tau\to 0$, that is,  as both the space discretization and the time discretization get finer and finer.
We know that, by stability requirements, the time step used in the time integration by a splitting method of  system (\ref{td.10}) must be chosen as $\tau < y_*/\rho(H)$,
where the stability threshold must verify $y_* \le 2 m$  for an $m$-stage splitting method. 
Since the Hermitian matrix $H$ comes from the space discretization of an unbounded self-adjoint operator, the spectral radius $\rho(H)$ will tend to infinity as $N\to \infty$, and thus inevitably $\tau\to 0$. It 
seems then reasonable to introduce the parameter 
\begin{eqnarray}
\label{tau}
  \theta \equiv \tau \rho(H)
\end{eqnarray}
 and analyze
the time-integration error  corresponding to a fixed value of $\theta < y_*$.

The fact that, for each $y \in (-y_*,y_*)$,  (\ref{Ophi(x)}) holds with $Q(y)$ given by (\ref{Q(y)}) implies that,
 for each $n\geq 1$,
\begin{eqnarray*}
     K(y)^n=
\left(
 \begin{array}{cc}
    \cos (n\phi(y)) + \epsilon(y) \sin(n \phi(y)) &   \gamma(y) \sin(n \phi(y))  \\
\displaystyle -\frac{1+\epsilon(y)^2}{\gamma(y)} \sin (n \phi(y)) &  \cos(n \phi(y)) - \epsilon(y) \sin(n\phi(y))
 \end{array} \right).
\end{eqnarray*}
This will allow us to obtain rigorous estimates for the error of approximating  $e^{-i t H} u_0$ by applying $n$ steps of a splitting method with time-step $\tau=t/n$. Specifically, we have 
\begin{eqnarray}
\nonumber
\left(
 \begin{array}{c}
q_n \\ p_n
 \end{array} \right) &=&
  K(\tau \, H)^n
\left(
 \begin{array}{c}
q_0 \\ p_0
 \end{array} \right) \\
\label{E-error}
 &=& O(n \phi(\tau \, H)) \left(
 \begin{array}{c}
q_0 \\ p_0
 \end{array} \right)
+ E(\tau \, H) 
 \left(
 \begin{array}{c}
\sin(n \phi(\tau \, H))q_0 \\ \sin(n \phi(\tau \, H)) p_0
 \end{array} \right),
\end{eqnarray}
where $O(y)$ denotes the rotation matrix (\ref{O(y)}) and the $2\times 2$ matrix $E(y)$ is given by
\begin{equation}   \label{matrix.E}
  E(y) = 
\left(
 \begin{array}{cc}
     \epsilon(y) &   \gamma(y) -1 \\
\displaystyle -\frac{1+\epsilon(y)^2}{\gamma(y)}+1& - \epsilon(y)
 \end{array} \right),
\end{equation}
so that the following theorem can be stated.

\begin{theorem}
\label{th:error0}
  Given $u_0 = q_0+i p_0$, let $u_n=q_n+ i p_n$ be the approximation to $u(n\tau) = e^{-i\, n \tau\, H}u_0$ obtained by applying $n$ steps of length (\ref{tau})  of a splitting method with stability threshold $y_*$.
   Then  one has
  \begin{eqnarray*}
    \| u_n - u(n\tau) \| \leq (n\mu(\theta)+\nu(\theta)) \, \|u_0\|
  \end{eqnarray*}
(in the Euclidean norm), where
\begin{eqnarray*}
  \mu(\theta) = \sup_{0\leq y \leq \theta} |\phi(y)-y|, \qquad
  \nu(\theta) = \sup_{0\leq y \leq \theta} \|E(y)\|.
\end{eqnarray*}

\end{theorem}

\begin{proof}
From (\ref{E-error}), we can write
\begin{eqnarray*}
   \| u_n - u(n \tau) \| & = & \Big\| \! \left( \! \begin{array}{c}
   					q_n - q(t_n) \\
					p_n - p(t_n)
				        \end{array}	 \! \right) \! \Big\| \le \Big\| \big( O(n \phi) - O(n \tau H)\big) 
				           \left( \! \begin{array}{c}
				                           q_0 \\
				                           p_0
				                     \end{array} \! \right) \! \Big\|   \\
				        &  & + \, \|E(\tau H)\| \;       
				                     \Big\| \! \left( \! \begin{array}{c}
   					                   \sin (n \phi) q_0 \\
					                   \sin (n \phi) p_0
				                        \end{array}	\! \right) \! \Big\|,
\end{eqnarray*}				                        
where, for clarity, $\phi \equiv \phi(\tau H)$. For the first contribution we have
\begin{eqnarray*}
   \Big\| \big( O(n \phi) - O(n \tau H)\big) 
				           \left( \! \begin{array}{c}
				                           q_0 \\
				                           p_0
				                     \end{array} \! \right) \! \Big\|  &  = &  \| (\e^{-i n \phi}  - \e^{-i n \tau H}) u_0 \| 
	 =   \| \e^{-i n\tau\, H}(1- \e^{-i n(\phi -\tau H)})u_0 \| \\ 
	 &  \le & 	\|(1- \e^{-i n(\phi - \tau H)})\| \, \|u_0\|      
\end{eqnarray*}
since $H$ is Hermitian. Now
\begin{eqnarray*}
  \|(1- \e^{-i n(\phi - \tau H)})\| & = & \big\| i \, n \int_0^1 \e^{-i n (\phi - \tau H)s} \,  (\phi - \tau H) \, ds \big\| \\
    & \le & n \int_0^1 \|\e^{-i n (\phi - \tau H)s} \| \; \|\phi - \tau H\| \, ds = n  \, \|\phi - \tau H\|  \\
    & = & n \, 	\max_{1\leq j \leq N} |\phi(\tau\, \omega_j)-\tau \, \omega_j| \le n \, \mu(\theta).	                
\end{eqnarray*}   
As for the second contribution, one has 
\[
   \Big\| \! \left( \begin{array}{c}
   	            \sin (n \phi) q_0 \\
                     \sin (n \phi) p_0
             \end{array}	 \right) \! \Big\|   =  \|\sin(n \phi) u_0\| \le \|\sin(n \phi)\| \, \|u_0\| \le \|u_0\|,
\]
whereas
  \begin{eqnarray*}
    \|E(\tau\, H)\|= \max_{1\leq j \leq N} \|E(\tau \omega_j)\| \leq \nu(\theta). 
  \end{eqnarray*}
and thus the proof is complete.  
\end{proof}

Notice that the error estimate in the previous theorem does not guarantee that, for a given $t$, the error in approximating $\e^{-i\, t \, H}u_0$ by applying $n$ steps of the method is bounded as $\rho(H)\to \infty$. As
a matter of fact, since $\tau=t/n$ must satisfy the stability restriction $\theta = \tau\rho(H) <  y_*$, 
so that $n>t\, \rho(H)/y_*$, one has that 
$n$ (and hence the error bound above) goes to infinity as $\rho(H)\to \infty$. 
This 
 can be avoided
by estimating the error in terms of $\|H u_0\|$ in addition to $\|u_0\|$.
The assumption that $\|H u_0\|$ can be bounded uniformly as the space discretization parameter 
$N\to \infty$, implies that 
the initial state $\psi(x,0)$ of the continuous time dependent Schr\"odinger equation
is such that $\partial_x^2 \psi(x,0)$ is square-integrable. The converse will also be true for reasonable space semi-discretizations and a sufficiently smooth potential $V(x)$. 

More generally, the assumption that $\psi(x,0)$ has sufficiently high spatial regularity (together with suitable conditions on the potential $V(x)$) is related to the existence of bounds of the form
$\|H^k u_0\| \le C_k$
that hold uniformly as $\rho(H)\to \infty$. In this sense, it is useful to introduce the following notation: 
\begin{itemize}
\item Given $k\geq 0$, we denote for each $u \in \mathbb{C}^N$
\begin{eqnarray*}
  \|u\|_{k} :=\|H^{k} u\|.
\end{eqnarray*}
%

\item For a $m$-stage splitting method with stability threshold $y_*$, given $k\geq 0$ and $\theta \in [0,y_*)$ we denote
\begin{eqnarray}
\label{eq:mu}
  \mu_k(\theta) &=& \sup_{0\leq y \leq \theta} \left|\frac{\phi(y)}{y}-1\right|\, (\theta/y)^k, \\
\label{eq:nu}
  \nu_{k}(\theta) &=& \sup_{0\leq y \leq \theta} \|E(y)\|\, (\theta/y)^k.
\end{eqnarray}
Clearly, $\mu_k(\theta)$ and 
$\nu_{k}(\theta)$ are bounded if and only if the method is of order $r\geq k$.
\end{itemize}
We are now ready to state the main result of this section.
\begin{theorem}
\label{th:error}
  Given $u_0 = q_0+i p_0$ and $t \in \mathbb{R}^+$, let $n$ be such that $\tau=t/n=\theta/\rho(H)$ (with $\theta <  y_*$), and let  $u_n=q_n+ i p_n$ be the approximation to $u(t) = \e^{-i\, t\, H}u_0$ obtained by applying $n$ steps of length $\tau$ of a $r$-th order splitting method with stability threshold $y_*$. Then, for each $k\in [0,r]$, 
  \begin{eqnarray}
\label{error}
    \|u_n - u(t)\| \leq \frac{t \,\mu_{k}(\theta) \, \|u_0\|_{k+1}+\nu_{k}(\theta) \|u_0\|_{k}}{\rho(H)^k}.
  \end{eqnarray}
\end{theorem}
\begin{proof}
We proceed as in the proof of Theorem~\ref{th:error0}. First we bound 
   \begin{eqnarray*}
\|\e^{-i n\tau\, H}u_0- \e^{-i n\phi(\tau\, H)}u_0\| &\leq& 
\tau^{k+1} \, \|(\tau H)^{-k-1}(1- \e^{-i n(\phi(\tau \, H)-\tau H)})\| \, \|u_0\|_{k+1} \\
&\leq& \frac{t \theta^k}{\rho(H)^k} \|(\tau H)^{-k-1} (\phi(\tau\, H)-\tau \, H))\| \, \|u_0\|_{k+1},
   \end{eqnarray*}
with
\begin{eqnarray*}
 \|(\tau H)^{-k-1}(\phi(\tau\, H)-\tau \, H)\| =\max_{1\leq j \leq N} 
    |(\tau \omega_j)^{-k-1}(\phi(\tau\, \omega_j)-\tau \, \omega_j)| 
\leq \mu_k(\theta)/\theta^{k}.
\end{eqnarray*}
Then the second term in (\ref{E-error}) verifies
\begin{eqnarray*}
  \tau^k \, \|(\tau H)^{-k}E(\tau\, H) \sin(n\phi(\tau \, H)) H^{k}u_0\| \leq 
     \frac{\theta^k}{\rho(H)^k} \|(\tau H)^{-k} E(\tau\, H)\| \, \|u_0\|_k,
\end{eqnarray*}
 and
  \begin{eqnarray*}
    \|(\tau H)^{-k} E(\tau\, H)\|= \max_{1\leq j \leq N} \|(\tau \omega_j)^{-k} E(\tau \omega_j)\| \leq \nu_k(\theta)/\theta^k, 
  \end{eqnarray*}
from which (\ref{error}) is readily obtained.  
\end{proof}

Some remarks are in order at this point:
\begin{enumerate}
  \item Recall that the estimate in Theorem~\ref{th:lubich} shows the behavior of the space discretization error (of a spectral collocation method applied to the 1D Schr\"odinger equation) as the number $N$ of collocation points goes to infinity. Our estimate (\ref{error}) shows in turn the behavior of the time discretization error as  $N\to \infty$, provided that $\tau=\theta/\rho(H)$ with a \emph{fixed} $\theta<y^*$. In that case, it can be shown that $\rho(H)^{-1}\leq L\, N^{-2}$ uniformly for all $N$, and thus the error of the full discretization admits the estimate
   \begin{eqnarray*}
 \frac{1}{N^{2k}} \left(C (1+t) \displaystyle \max_{0 \le t' \le t} \, \left\|\partial_x^{2k+2} \psi(\cdot,t') \right\|
    +L \big( t\, \mu_{k}(\theta) \|u_0\|_{k+1}+  \nu_{k}(\theta) \|u_0\|_{k}\big) \right).
    \end{eqnarray*}
Notice the similarity of both the space and time discretization errors
($\|u_0\|_{k+1}=\|H^{k+1} u_0\|$ is a discrete version of a continuous norm $||\psi(\cdot,0)||_{k+1}$ which is equivalent to the Sobolev norm $\|\partial_x^{2k+2} \psi(\cdot,0) \|$).

\item Given a splitting method with stability threshold $y^*$ of order $r$ for the harmonic oscillator, consider $\mu_r(\theta)$ and $\nu_r(\theta)$ in  (\ref{eq:mu})-(\ref{eq:nu}) for a fixed $\theta<y^*$. 
If instead of analyzing the behavior as $N$ increases of the time discretization error committed by the
splitting method when applied with $\tau=\theta /\rho(H)$, 
one is interested in analyzing the error with fixed $H$ and \emph{decreasing} $\tau \leq \theta/ \rho(H)$, 
one proceeds as follows. Since by definition 
\begin{eqnarray*}
  \mu_r(\tau \rho(H) ) \leq \mu_r(\theta) \left(\frac{\rho(H) \tau}{\theta}\right)^r, \quad
 \nu_r(\tau \rho(H) ) \leq \nu_r(\theta) \left(\frac{\rho(H) \tau}{\theta}\right)^r, 
\end{eqnarray*}
then reasoning as in the proof of Theorem \ref{th:error},
one gets the estimate
\begin{eqnarray*}
      \|u_n - u(t)\| \leq \frac{t \,\mu_{r}(\theta) \, \|u_0\|_{r+1}+\nu_{r}(\theta) \|u_0\|_{r}}{\theta^r}\, \tau^r.
\end{eqnarray*}
\end{enumerate}  

From a practical point of view, (\ref{error}) can be used to obtain \textit{a priori} error estimates just by replacing the exact $\rho(H)$ by an approximation (obtained for instance with some generalization of the power method), or by an estimation based on the knowledge of bounds of the potential and the eigenvalues of the discretized Laplacian.

The error estimates in Theorem~\ref{th:error} provide us appropriate criteria to construct splitting methods to be applied for the time integration of systems of the form (\ref{td.10}) that result from the spatial semi-discretization of the time dependent Schr\"odinger equation. Such error estimates suggest in particular that different splitting methods should be used depending on the smoothness of 
initial state in the original equation. Also, Theorem~\ref{th:error} indicates that for sufficiently long time integrations, the actual error will be dominated by the phase errors, that is, the errors corresponding to $\mu_{k}(\theta)$.

\section{On the construction of new symplectic splitting methods}
\label{sec.4}

Observe that when comparing the error estimates in Theorem~\ref{th:error} for a given $k \ge 0$ 
corresponding to two methods with different number of stages $m$ and $m'$ respectively, one should consider time steps $\tau$ and $\tau'$ that are proportional to $m$ and $m'$ respectively.
In this way, the same computational effort is needed for both methods to obtain a numerical approximation of $u(t)$ for a given $t>0$. It makes sense,
then, to consider a {\em scaled time step} of the application with time step $\tau$ of a $m$-stage splitting method to the system (\ref{td.10}). This can be defined as 
\begin{equation}  \label{thetaprime}
   \theta' \equiv \frac{\theta}{m}  =   \frac{\tau \rho(H)}{m},
\end{equation}   
so that
the relevant error coefficients associated to the error estimates in Theorem~\ref{th:error} are  
$\mu_{k}(\theta' m)$ and $\nu_{k}(\theta' m)$.

The task of constructing a splitting method in this family can be thus precisely formulated as follows.

\

\noindent \textbf{Problem}. \textit{Given a fixed number $m$ of stages in (\ref{prod-gen2p2}), and for prescribed values of $k \ge 0$ and scaled time step $\theta' \in (0,2)$, design some splitting method having order $r\geq k$ and stability threshold $y_* > \theta' m$, which tries to optimize the main error coefficient $\mu_{k}(\theta' m)$ while keeping $\nu_{k}(\theta' m)$ reasonably small}. 

\

We have observed, however, that trying to construct such optimized methods in terms of the coefficients of the polynomial entries $K_j(y)$ ($j=1,2,3,4$) of the propagation matrix $K(y)$ leads us to very ill-conditioned systems of algebraic equations. That difficulty can be partly overcome by taking into account the following observations: 
\begin{itemize}
\item The functions $|\phi(y)/y-1|$ and $\|E(y)\|$ ($y \in (-y_*,y_*)$) determining the error estimates in Theorem~\ref{th:error} uniquely depend on two polynomials: the stability polynomial $p(y)$ given in (\ref{stab-pol}) and 
  \begin{eqnarray}
\label{qpol}
    q(y) = \frac{K_2(y)-K_3(y)}{2}.
  \end{eqnarray}
Indeed, from one hand, $\phi(y)=\arccos(p(y))$, so that $|\phi(y)/y-1|$ uniquely depends on $p(y)$. On the other hand, according to Proposition~\ref{p:2}, 
\begin{eqnarray*}
  q(y) = (1+\frac12 \delta(y))\, \sin(\phi(y)), \quad
\mbox{where} \quad
  \delta(y) = \left( \gamma(y) + \frac{1+\epsilon(y)^2}{\gamma(y)}\right)-2, 
\end{eqnarray*}
and one can get by straigthforward algebra that $\|E(y)\|$ is (in Euclidean norm)  
\begin{eqnarray*}
 \|E(y)\| = \sqrt{\delta(y) \left( 1+ \frac{\delta(y)}{2}+\sqrt{\delta(y)+ \frac{\delta(y)^2}{4}}\right)}
\end{eqnarray*}
(and thus $\|E(y)\|=\sqrt{\delta(y)} + \mathcal{O}(\delta(y))$ as $\delta(y) \to 0$).
\item Given an even polynomial $p(y)$ and an odd polynomial $q(y)$, there exist a finite number of propagation matrices $K(y)$ such that (\ref{stab-pol}) and (\ref{qpol}) hold.
Indeed, the entries of such stability matrices are of the form
\begin{eqnarray*}
  K_1(y) = p(y)+d(y), \qquad
  K_2(y) = q(y)+e(y), \\
  K_3(y) = -q(y)+e(y), \qquad
  K_4(y) = p(y)-d(y),
\end{eqnarray*}
where $d(y)$ and $e(y)$ are respectively even and odd polynomials satisfying 
\begin{eqnarray}
\label{pqde}
  p(y)^2+q(y)^2-1 = d(y)^2+e(y)^2.
\end{eqnarray}
It is not difficult to see that there is a finite number of choices for such polynomials $d(y)$ and $e(y)$. The ill-conditioning mentioned before seems to come mainly from the ill-conditioning of the problem of determining $d(y)$ and $e(y)$ from prescribed polynomials $p(y)$ and $q(y)$. Obviously, a necessary condition for the existence of such polynomials $d(y)$ and $e(y)$ with real coefficients is that $p(y)^2+q(y)^2-1 \geq 0$ for all $y$. It is also straightforward to see that, for a $r$th order method, $d(y)=\mathcal{O}(y^{r+1})$ and $e(y)=\mathcal{O}(y^{r+1})$ (as $y\to 0$), and thus
\begin{eqnarray}
  \label{p2q2}
  p(y)^2+q(y)^2-1=\mathcal{O}(y^{2r+2}). 
\end{eqnarray}
In addition, if the method has stability threshold $y_*>0$, then there exists $0<y_1<\cdots<y_l<y_*$ such that $\phi(y_j)=j \pi$, and thus
\begin{eqnarray}
\label{stab-cond}
  p(y_j) = (-1)^j, \quad p'(y_j)=0, \quad q(y_j)=0, \quad \mbox{for} \quad j=1,\ldots,l.
\end{eqnarray}

\end{itemize}

For simplicity, we restrict ourselves to the construction of $m$-stage methods of even order $r$ that are intended to have small values of $\mu_{r}(\theta' m)$ and $\nu_{r}(\theta' m)$ for a prescribed scaled time
step $\theta' \in (0,2)$. When designing such a method, we follow several steps:
\begin{enumerate}
\item First find two polynomials $p(y)$ and $q(y)$ with small value of 
$\mu_{r}(\theta' m) + \lambda \nu_{r}(\theta' m)$ (for some $\lambda<1$)
among those satisfying the following three conditions:
  \begin{enumerate}
  \item There exist $y_j \approx j \pi$ $(j=1,\ldots,l)$ with $l\, \pi \leq \theta' m \leq (l+1)\pi$ such that (\ref{stab-cond}) holds;
  \item $p(y) = \cos(y) + \mathcal{O}(y^{r+1})$,  $q(y) = \sin(y) + \mathcal{O}(y^{r+1})$,
and (\ref{p2q2}) as $y\to 0$; 
  \item $p(y)^2+q(y)^2-1>0$ for all $y \in \mathbb{R}$.
  \end{enumerate}
\item Find all possible pairs $(d(y),e(y))$ of real (even and odd respectively) polynomials satisfying $(\ref{pqde})$, and for each pair $(d(y),e(y))$, construct the corresponding $2\times 2$ matrix $K(y)$.
\item Apply the algorithm given in \cite{{blanes08otl}} to each of the matrices $K(y)$ obtained in the previous step. In this way we will get the vector of coefficients $(a_i,b_i)$ of all the splitting methods having a progagation matrix $K(y)$ satisfying (\ref{stab-pol}) and (\ref{qpol}) for the pair of polynomials $(p(y),q(y))$ determined in the first step. Since Theorem~\ref{th:error} gives exactly the same error estimate (\ref{error}) for all the splitting schemes obtained in that way, we choose (with the aim of reducing the effect of round-off errors) one that minimizes 
  \begin{eqnarray*}
    \sum_{j=1}^{m} (|a_j|+|b_j|).
  \end{eqnarray*}
\end{enumerate}

This procedure has been applied to construct several splitting methods of different orders $r$, number of
stages $m$ and scaled time steps $\theta'$. We collect in Table \ref{tab:1} the relevant
parameters of some of them, whereas the actual coefficients $a_j,b_j$ can be found at 
\texttt{www.gicas.uji.es/Research/splitting1.html}. As a matter of fact, all the methods have $m+1$ pairs
of coefficients $a_i$, $b_i$, but $b_{m+1}=0$. In consequence, the last stage at a given step can be
concatenated with the first one at the next step, so that the overall number of stages is $m$. This property
is called FSAL (first-same-as-last) in the numerical analysis literature.
According to the previous comments, the new schemes are aimed at integrating equation (\ref{td.10})
under very different conditions of regularity.

\begin{table}[t]
\begin{eqnarray*}
  \begin{array}{|ccccccc|} \hline
m & r & \theta' & y_*/m & \sum_{j} (|a_j|+|b_j|) & \mu_r(\theta' m) & \nu_r(\theta' m) \\ \hline
10 &  6 &  1 &  1.1617 &   4.022 &  0.0009341 &  0.0372 \\
  20 &  16 &  1 &  1.0456 &  3.0553 & 0.000611028 & 0.0258433 \\
  30 &  24 &  1 & 1.0246 & 3.19658 &  0.0000841871 & 0.0373544 \\
 30 &  6   & 1.4 &  1.41876 &   3.0921 &  0.0000518519 &  0.0131295 \\ 
  30 &  0   &  1 & 1.1411 & 3.04948 & 2.91902\cdot 10^{-13} & 2.28673\cdot10^{-9}\\ 
  30 &  0   &  0.75 & 1.027 & 3.44381 & 1.2545\cdot 10^{-17} & 5.96706\cdot10^{-14}\\    
30 &  0   &  0.5 & 0.937874 & 3.84442 & 7.96031\cdot 10^{-24} & 6.66693\cdot10^{-18}\\ 
  40 &  0   &  1 &  1.15953 &  3.21986 &  1.06301 \cdot 10^{-15} &  1.07587\cdot 10^{-12} \\ \hline  
  \end{array}
\end{eqnarray*}
 \caption{Relevant parameters of several new splitting methods of order $r$
especially designed to integrate with scaled time step $\theta' = \tau \rho(H)/m$ the semi-discretized Schr\"odinger equation.  Here $\rho(H)$ is the spectral radius of the matrix $H$, $m$ is the number
 of stages, $y_*$ stands for the stability threshold and $\mu_r(\theta)$, $\nu_r(\theta)$ are the coefficients appearing in the error estimate (\ref{error}).}
  \label{tab:1}
\end{table}

The first three methods in Table~\ref{tab:1} are designed to be applied with the same scaled time-step $\theta'=\tau\rho(H)/m=1$, and thus the three of them have the same computational cost.
The order $r$ of the methods is increased by adding more stages, while keeping reasonably small error coefficients $\mu_r(\theta)$ and $\nu_r(\theta)$. This will be advantageous, according to Theorem~\ref{th:error}, for sufficiently regular initial states. Alternatively, one may want to use the additional number of stages to reduce the computational cost while keeping the same order $r=6$. This can be illustrated with the fourth method in Table~\ref{tab:1}, which has been optimized for scaled time-step $\theta'=1.4$, and thus is substantially cheaper than  the first method (optimized for $\theta'=1$), while having smaller error coefficients $\mu_6(\theta)$ and $\nu_6(\theta)$.

We now turn our attention to the methods of order zero in Table~\ref{tab:1}, which according to Theorem~\ref{th:error}, are the methods of choice for very low regularity conditions.
Although they have comparatively smaller error coefficients than the methods of order $r>1$, one should bear in mind that the error estimates (\ref{error}) for $r \geq k >1$ decrease with $\rho(H)^{-k}$ as the spectral radius $\rho(H)$ increases, while for methods of order $r=0$ the same estimate holds independently of the size of $\rho(H)$. Comparing the first three methods of order $r=0$ and $m=30$, we see that, not surprisingly, the accuracy of the methods can be improved by increasing the computational cost (by considering methods optimized for lower values of $\theta'$). This is analogous to increasing the accuracy of the application of a Chebyshev polynomial of degree $m=30$ by decreasing the time-step size. Now, by comparing the first $30$-stage method of order $0$ with the method with $m=40$ and order $0$, we see that the accuracy can be increased also by increasing the number of stages from $m=30$ to $m=40$ while keeping the same computational cost (with $\theta'=1$). This is similar to increasing the accuracy of Chebyshev approximations, while keeping the same computational cost, by increasing the degree of the polynomial from $m=30$ to $m=40$.

Recall that, if $\omega_1,\ldots,\omega_N$ are the eigenvalues of $H$, numerically integrating  (\ref{td.10}) by a splitting method is mathematically equivalent to applying the splitting method to $N$ uncoupled harmonic oscillators with frequencies $\omega_j$. Particularizing the proof of Theorem~\ref{th:error} 
to this case, it is quite straightforward to conclude that, 
when integrating the system with scaled time step $\theta'$ (that is, with $\tau=m \theta'/\rho(H)$, where $m$ is the number of stages of the scheme,)
 the relative error made in each oscillator can be bounded by
\begin{eqnarray*}
  t \, |\omega_j| \, \mu_j +\nu_j,
\end{eqnarray*}
where
\begin{eqnarray*}
  \mu_j = \left| \frac{\phi(m y_j)}{(my_j)}-1 \right|, \quad \nu_j = \|E(m y_j)\|\quad
\mbox{with} \quad |y_j| = \frac{\tau \, \omega_j}{m} = \frac{\theta' \, \omega_j}{\rho(H)}
\end{eqnarray*}
and thus $|y_j|\leq \theta'$.
When such a system of harmonic oscillators originates from a continuous problem 
possessing a high degree of regularity,
the highest frequency oscillators have much smaller amplitude and thus can be approximated less accurately than the lower frequency oscillators without compromising the overall precision. For lower regularity conditions, the overall precision will be more affected by the accuracy of the approximations corresponding to higher frequency oscillators.

\begin{figure}[!ht]
  \centering
  \includegraphics[scale=1.4]{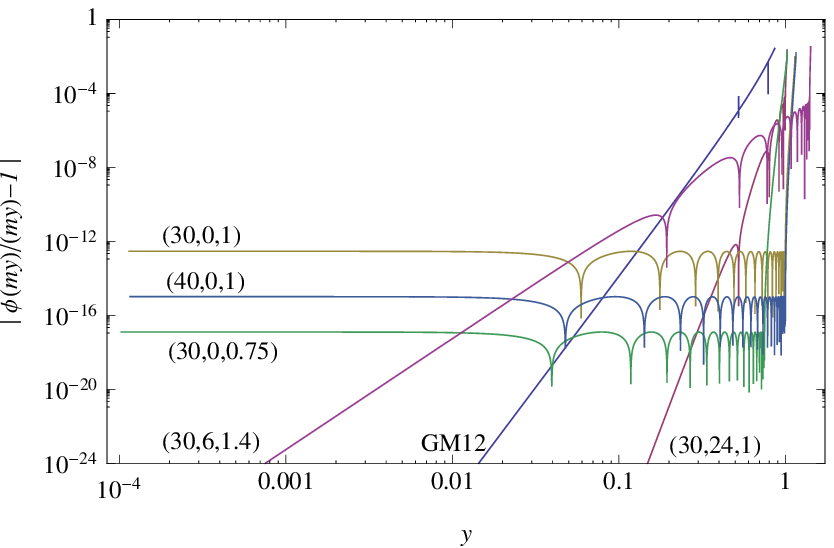}
  \includegraphics[scale=1.4]{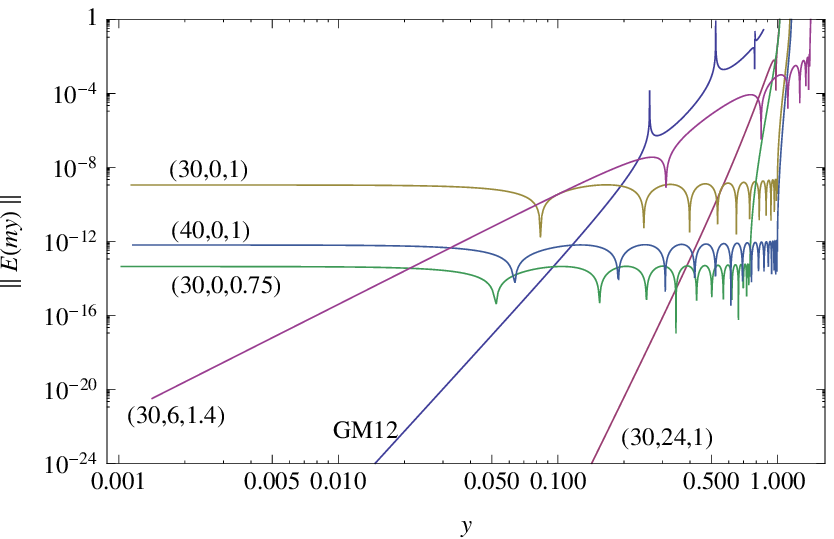}
  \caption{Graphs of $|\phi(m y)/(my)-1|$ (top) and  $\|E(m y)\|$ (bottom) for some of the $m$-stage splitting
  methods collected in Table \ref{tab:1} and the 12th-order scheme
$\mathrm{GM}12$. Each new splitting method is identified by the triad $(m,r,\theta')$, indicating its
number of stages $m$, order $r$ and scaled time step $\theta'$, as in Table \ref{tab:1}.
The resonances associated with the instability of $\mathrm{GM}12$
at $y \approx k \pi$ ($k=2,3$) are visible, especially in the second graph.}
  \label{fig:1}
\end{figure}

With the aim of illustrating the relative error made in each harmonic oscillator, in Figure~\ref{fig:1} 
we represent  (in double logarithmic scale) $|\phi(m y)/(my)-1|$ and $\|E(m y)\|$ (which are even functions of $y$) for
$ y \in [0,\theta']$ 
 for some of the methods collected in Table~\ref{tab:1}, identified by appropriate labels indicating their
 respective number of stages, order and scaled time step $(m, r, \theta')$.
Observe that both functions of $y$ exhibit a similar behavior for each splitting method,
although the values taken by the second one are several orders of magnitude smaller, since the methods
are designed to minimize mainly the phase error coefficient $\mu_k(\theta)$. The order $r$ of each of the methods is reflected in the slope of the curves as $y$ approaches $0$.

We also include in Figure~\ref{fig:1} the $12$th order $12$-stage scheme presented in~\cite{gray96sit}, which is perhaps the most efficient when applied to harmonic oscillators among those (non-processed) splitting methods currently found in the literature. We denote it by $\mathrm{GM}12$. It has a relative stability threshold $y_*/12 = 0.2618$,
so that, strictly speaking, it should be used with $\theta' = \tau \rho(H)/12< 0.2618$ to guarantee stability. However, 
it seems in practice that the method can be safely used with $\theta' = 0.932$ (for a larger value of the scaled time step $\theta'$, the method becomes very
unstable), because $||p(y)|-1| < 10^{-6}$ provided that $|y|/12< 0.932183$. 
The theoretical instability for $\theta' \in (0.2618,0.932183)$ is only relevant after
a very large number of steps, and reveals itself as 
resonance peaks (which are clearly visible in the graph of $\|E(my)\|$ in Figure~\ref{fig:1} for $k=2,3$) near the values $\theta' = k \pi/12$, $k=1,2,3$.

We can see in Figure~\ref{fig:1} that $\mathrm{GM}12$ is less accurate than the $30$-stage $24$th order method for the whole frequency range, and thus the former will show a poorer performance than the later for any regularity conditions. If the amplitudes at higher frequencies decrease fast enough, the $24$th order method will give very accurate approximations of  $u(t) = \e^{-i t H} \, u_0$ at a relatively low cost (since $\theta'=1$). The $6$th order method with $m=30$ stages gives correct approximations for all harmonic oscillators within the range $y \in [-1.4,1.4]$, and hence it is expected to give excellent approximations under mild regularity with a comparatively lower cost than the $24$th order method ($\theta'=1.4$ compared with $\theta'=1$).
Clearly, the methods of order $0$ will be the right choices for low regularity conditions, since the corresponding phase errors $|\phi(m y)/(m y)-1|$ are uniformly bounded for all $y \in [-1,1]$. Among them, the method with $m=30$ and $\theta'=1$ can be accurate enough in many practical computations. If more precision is required, one can either consider the method with $\theta'=0.75$ (with result in a 25\% increase of the computational cost), or use the method with $m=40$ and $\theta'=1$, without any increase of the computational cost (at the expense of having less frequent output).

\section{Numerical examples}
\label{sec.5}

The purpose of this section is twofold. On the one hand, since the symplectic splitting
methods we have presented here to approximate $\e^{-itH}u_0$ involve only
products of the matrix $H$ with vectors, it makes sense to compare them with
other well established schemes of this kind, such as the Chebyshev and
Lanczos methods. Although a thorough comparison with the family of 
splitting methods proposed in this work will be the subject of a 
forthcoming paper \cite{blanes10splitt}, we include here some results which show that the new schemes are
indeed competitive for evaluating $\exp(-i t H) u_0$,
at least in the example considered.

On the other hand, since 
Theorem~\ref{th:error} provides a rigorous \emph{a priori}
estimate on the error committed  when using a splitting method of
the form (\ref{prod-gen2p2}) in the time integration of equation
(\ref{td.1}), it is interesting to check how this theoretical error estimate behave
in practice for some of the methods constructed here.

\subsection{A preliminary comparison with Chebyshev and Lanczos methods}

As is well known, Chebyshev and Lanczos methods provide high order polynomial approximations
to $\e^{-itH}u_0$ requiring only matrix-vector products. The former is neither unitary nor 
symplectic, whereas the later is unitary, but
symplectic only in the Krylov subspace (which changes from one time step to the next).

To carry out this comparison we choose the very simple example previously considered in
\cite{lubich08fqt}. The problem consists in approximating
$y= \e^{-iA}v$, where $v$ is a random vector
of unit norm and $A$ is the tridiagonal matrix 
$A=\frac{\omega}{2} \, \mbox{tridiag}(-1,2,-1)$ of dimension 10000.
The eigenvalues of $A$ are contained in the interval $[0,2\omega]$. 

We have
implemented the Chebyshev and Lanczos algorithms in the usual
way (see \cite{lubich08fqt}) with the particularity that, since the
range of values for the eigenvalues is known, both the Chebyshev and the new
splitting methods are used with a shift to the midpoint of the spectrum. In other
words, $y = \e^{-i\omega I} \, \e^{-i(A-\omega I)}v$, with $I$
the identity matrix. This shift allows us to take $\rho(A- \omega I)\simeq
\omega$. 

\begin{figure}[h!]
 \centering
  \includegraphics[height=12cm,width=14cm]{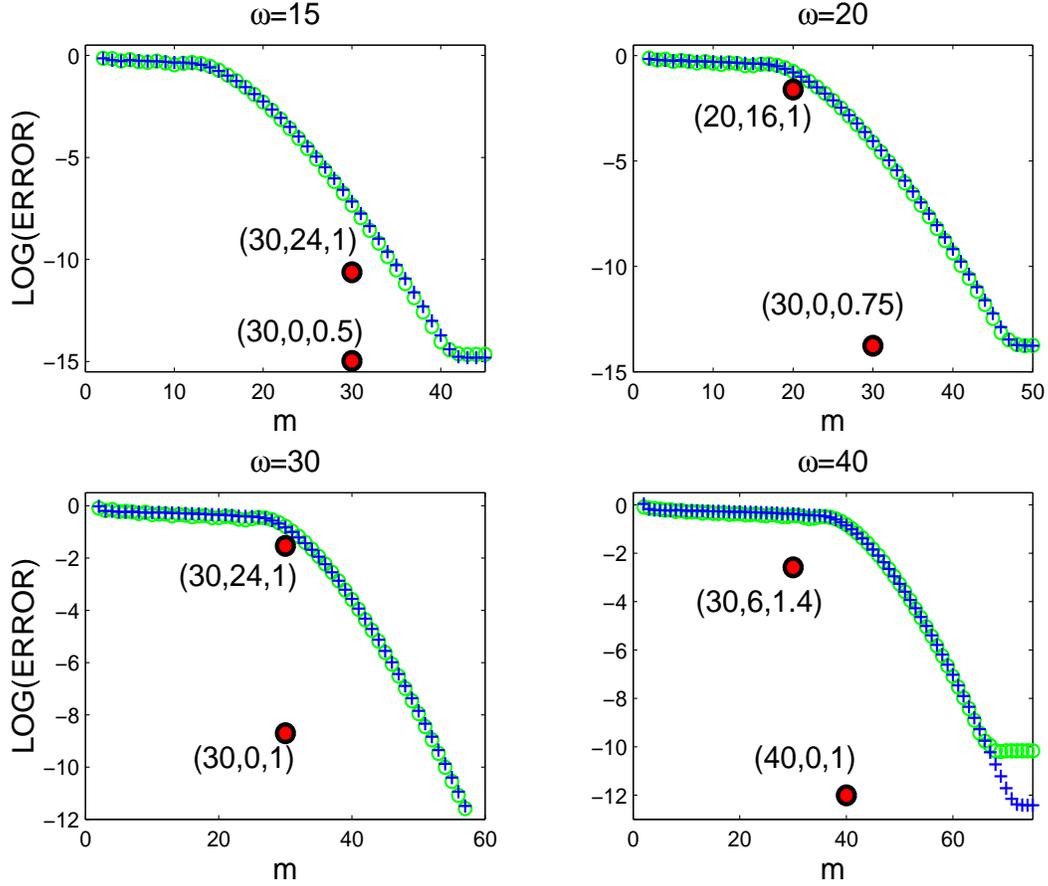}
\caption{{
 Error, $\|y-y_{ap}\|$, versus degree of the polynomial, $m$, for approximations
to $y=\e^{-iA}v$ where $v$ is chosen as a random vector of unit
norm and $A$ is the tridiagonal matrix 
$A=\frac{\omega}{2} \, \mbox{tridiag}(-1,2,-1)$ of dimension 10000.
Here $y$ is computed numerically to high accuracy and $y_{ap}$
correspond to the approximate solutions obtained by each method.
Results corresponding to Lanczos (crosses), Chebyshev 
(small circles) and several new splitting 
$(m,r,\theta')$ methods (big circles) are depicted. }}
 \label{figLubich}
\end{figure}

Figure~\ref{figLubich} shows the error, $\|y-y_{ap}\|$ for
different degrees $m$ of the polynomials used and for
$\omega=15,20,30,40$. Here $y$ is computed
numerically to high accuracy and $y_{ap}$ corresponds to the
approximate solution obtained by each scheme. Each particular value of
$m$ in the Lanczos and Chebyshev methods corresponds to a different
$m$th-order polynomial approximation (denoted by small crosses and
circles, respectively). We clearly observe that, for this
irregular problem, the Lanczos method converges to the optimal
Chebyshev method, the main difference between both schemes being 
the number of vectors to be kept in memory.

To apply the new splitting methods, we notice that for this problem the
time step $\tau=1$, and the corresponding spectral radius
$\rho\simeq \omega$. Therefore we shall consider splitting
methods whose value of $\theta'$ given by (\ref{thetaprime}) satisfies
\[
  \frac{\tau \rho}{m}\simeq \frac{\omega}{m} \leq \theta'.
\]
In other words, for each $\omega$ the method $(m,r,\theta')$ is such that
$m \, \theta'\geq \omega$. 

From the graphs of Figure~\ref{figLubich}, it is clear that, for
each value of $\omega$, we can always select one particular 
splitting method (big dots) which outperforms both Chebyshev
and Lanczos. High-order splitting methods show a worst
performance than schemes of order zero for this problem, since they require typically a
higher degree of regularity. 


\subsection{The P\"oschl--Teller potential}

We next illustrate the error estimate provided by
Theorem~\ref{th:error} for the class of splitting methods proposed here.
We also compare the error in the time integration
with the error coming from the space discretization for different values of the mesh size $N$. 
For that purpose we choose a well known
anharmonic quantum potential leading to analytical solutions and
consider, for clarity, only the 30-stage splitting methods of
order six and order zero collected in Table \ref{tab:1}. Specifically, we
consider the P\"oschl--Teller potential 
 \[
  V(x) = -\frac{\alpha^2}{2\mu}  \frac{\lambda(\lambda-1)}{\cosh^2(\alpha x)},
 \]
with $\lambda > 1$. It has been frequently used in polyatomic
molecular simulation and is also of interest in
supersymmetry, group symmetry, the study of solitons, etc.
\cite{shi07fmi,flugge71pqm,lemus02cot}. The parameter $\lambda$ gives the
depth of the well, whereas $\alpha$ is related to the range of the
potential. The energies are
 \[
 E_n =-\frac{\alpha^2}{2\mu} (\lambda - 1 - n)^2, \qquad \mbox{ with }  \; 0\leq n \le \lambda - 1.
 \]

We take the following values for the parameters (in atomic units, a.u.): the
reduced mass $\mu= 1745$ a.u., $\alpha=2, \, \lambda=24.5$
(leading to 24 bounded states), $x\in[-5,5]$, and assume the
system is periodic. The periodic potential is continuous and very
close to differentiable. For $N=128$ we have $\rho(H)\simeq
0.635$, whereas for $N=256$ one gets $\rho(H)\simeq 1.85$.

We take as initial condition the Gaussian function, $\psi
(x,0)=\sigma \, \e^{-b^2x^2}$, where $\sigma$ is a normalizing
constant. With $b=3$ the function and all its derivatives of
practical interest vanish up to round off accuracy at the
boundaries. The initial conditions contain part of the continuous
spectrum, but this fact does not cause any trouble due to the
smoothness of the periodic potential and wave function. As an
illustration,  some of the corresponding values of $\|u_0\|_k$ for
$N=128$ are: $\|u_0\|_1=0.629909$, $\|u_0\|_6=0.0722513$,
$\|u_0\|_7=0.0478388$. They decrease only moderately for the first
values of $k$ (before they increase again due to the
contributions coming from higher energies). The corresponding values for $N=256$ 
are quite similar
to the previous ones. For this problem, both large and
very small spatial errors are expected from spectral methods,
depending on the mesh employed. It is then useful to have different
methods with large values of $\theta'$ when low accuracy is
desired and smaller values of  $\theta'$ for high accuracy.

We integrate for $t\in[0,128 \,T]$ with $T=333$ and measure the
2-norm error in the discrete wave function, $\|u_{ex}(2^i T)-u_{ap}(2^i
T)\|$, for $i=0,1,\ldots,7$. The values $u_{ex}$ are
computed using the same spatial discretization and an accurate
time integration (using a very small time step), whereas $u_{ap}$ stand for
the numerical approximations obtained with splitting methods. For a
given spatial discretization, we choose the time step for each of
the new methods such that $\tau\leq m \, \theta'/\rho(H)$. In consequence, a
period $T$ has to be divided into $M$ steps such that $M=T/\tau\geq
T\rho(H) / (m \theta')$. In particular, for the 6th-order method $(30,6,1.4)$
and $N=256$, since
$\rho(H)\simeq 1.85$, we take  $M=15\geq (333\times 1.85)/ (30
\times 1.4)$ (each period $T$ requires 15 steps, 450 stages or
1800 FFTs calls).


The results obtained are shown in Figure~\ref{fig3PT}. Solid lines
represent the error with respect to the exact solution for the
same spatial discretization, whereas dashed lines correspond to
the total error with respect to the exact solution obtained with a
finer mesh (it is the sum of the spatial error and the error from
the time integration). Dotted lines are obtained with the estimate
(\ref{error}).

\begin{figure}[t!]
  \centering
   \includegraphics[scale=0.75]{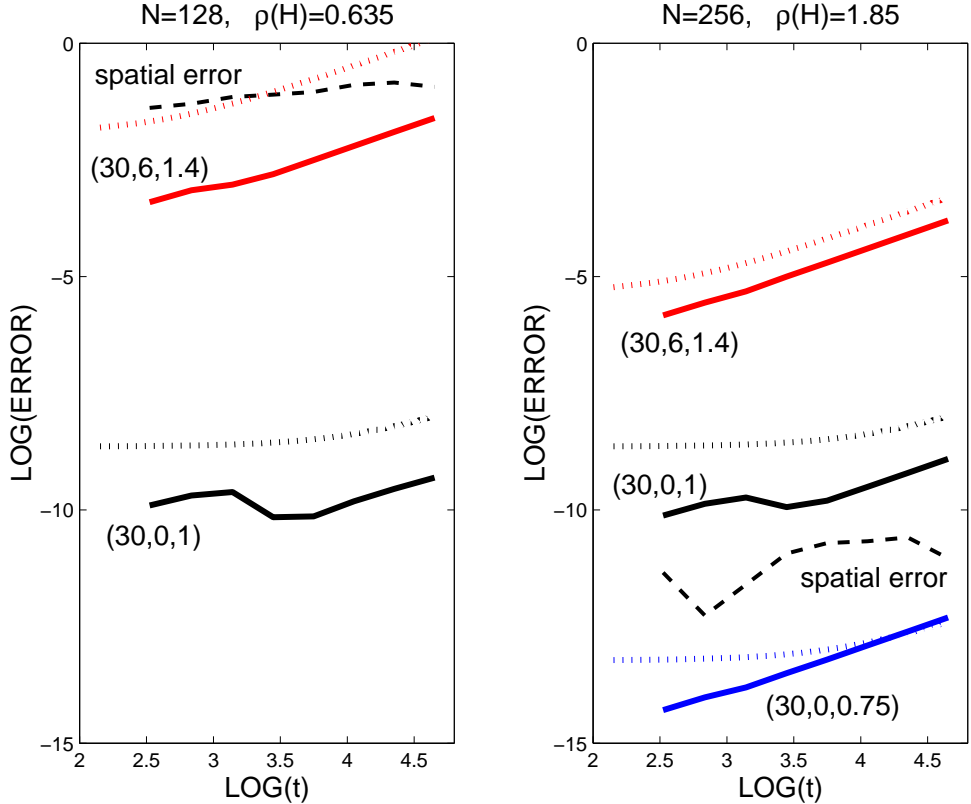}
  \caption{
Error in time integration (solid lines), error bounds from
(\ref{error}) (dotted lines) for methods $(30,6,1.4)$ and
$(30,0,1)$, and the spatial error (dashed lines), along the
interval $t\in[0,128 \, T]$ for the P\"oschl--Teller potential. We
have also included in the right panel the results obtained with
the 30-stage method of order zero and $\theta'=0.75$, 
$(30,0,0.75)$.}
  \label{fig3PT}
\end{figure}

We observe that the spatial error decreases exponentially with $N$
due to the smoothness and periodicity of the problem. To estimate this spatial error we take
the results obtained with $N=512$ and an accurate time integration
as the exact solution, and compare with the solution computed up
to a high accuracy for $N=128$ and $N=256$. For $N=128$ the
spatial error dominates the total error, so that the most
convenient time integration scheme is one able to provide such
accuracy with a large time step. These requirements are fulfilled
by the $(30,6,1.4)$ method, especially designed to be used with
$\theta' = 1.4$, whereas  scheme $(30,0,1)$ gives us higher accuracy than
necessary and with more computational cost. Method
$(30,6,1.4)$ can be used with a time step $\tau$ about a 40$\%$
larger than method $(30,0,1)$, and thus its computational cost is reduced approximately 
by this factor .

However, for $N=256$ the spatial error reaches nearly round off
accuracy, and it could be convenient to employ methods able to
provide this accuracy with the minimal computational cost. Notice
that in this case  the error committed by the 30-stage method with
$\theta'=1$,  $(30,0,1)$, is still larger than the spatial error.
In consequence, it makes sense integrating in time with a method specially
designed to be used with a smaller time step. Thus, in particular,
we reach round off accuracy with the 30-stage
method  $(30,0,0.75)$ ($\theta'=0.75$) which is nearly twice more
expensive than the method with $\theta'=1.4$.

 We have also performed here the time integration with the 12th-order
scheme GM12, which in the case of $N=128$ requires a
scaled time step of $\theta'=0.49$ to give a precision similar to
that obtained by $(30,6,1.4)$ with $\theta'=1.4$. With $N=256$,
GM12 must be applied with $\theta'=0.19$ to achieve the
precision obtained by  $(30,0,0.75)$ with $\theta'=0.75$, thus
requiring approximately four times more FFT calls.


\section{Concluding remarks}
\label{sec.6}

The time integration of the Schr\"odinger equation previously
discretized in space has been extensively studied in the
literature. This is essentially equivalent to approximate 
$u(t)=\e^{-itH}u(0)$, where $H$ is a real symmetric matrix and $u(0)$ 
represents the discrete wave function. In this work we propose 
using symplectic splitting integration
methods to get this approximation.  The main difference with
standard polynomial approximations is that in the products
$Hv=H\,\mbox{Re}(v)+i \, H\,\mbox{Im}(v)$, the real and imaginary parts are computed
sequentially instead of simultaneously (i.e., the computation of
the real part is used in the computation of the imaginary part and
vice versa in consecutive stages). These schemes are
conjugate to unitary methods, so that  the errors in norm and
energy do not grow secularly \cite{blanes08otl}.

To carry out the integration, one divides the whole time interval into $n$ steps
of length $\tau=t/n$ and applies an $m$-stage method at each time
step. The total computational cost of the method is measured by the product $n \, m$
instead of $m$. The analysis carried out in this paper allows us, in particular,
to construct a particular symplectic splitting scheme of the
form (\ref{compos}) which minimizes the total cost $n \, m$, given a
a prescribed tolerance, the
spectral radius $\rho(H)$ of the corresponding Hamiltonian matrix
$H$ and the norm of its action on the initial condition, $\|H^k
u(0)\|$.  We
have observed that the optimal methods in this sense have relatively
large values of $m$. We can choose the most appropriate method for
each problem, i.e. the method, $(m,r,\theta')$, with the largest
value of $\theta'$ which provides the desired accuracy for a given
problem.

The error analysis of splitting methods provided here allows one
to get \textit{a priori} bounds on the propagating error when
numerically integrating with a given time step which are
comparable to similar estimates for the space discretization
error. Moreover, it permits to construct new classes of schemes
with a large stability interval specifically designed to be used
with a certain (large) time step in such a way that the accuracy
is similar to the spatial discretization error for a given space
regularity. The main ingredients in the process are again the values of
$\rho(H)$, $\|H^k u_0\|$, and the estimate provided by
Theorem~\ref{th:error}. The numerical examples considered
illustrate the validity of our approach. In particular, they show that 
there are methods in this family which are competitive with other
standard procedures, such as Chebyshev and Lanczos methods.

By following this procedure it is indeed possible to generate a
list of integration schemes specifically designed to be used under
different regularity conditions on the initial state and the
Hamiltonian matrix which involve in each case an error comparable
to that coming from the spatial discretization. It is our purpose
in a forthcoming paper \cite{blanes10splitt} to elaborate an
algorithm in such a way that, given a prescribed tolerance, an
initial state $u_0$ and a Hamiltonian matrix $H$, automatically
selects the most efficient time integration method in this family
fulfilling the requirements supplied by the user. Moreover, we
will also carry out a detailed numerical study of this family of
splitting methods and the proposed automatic algorithm in
comparison with the
Chebyshev polynomial expansion scheme and the
Lanczos iteration method.

\subsection*{Acknowledgements}

This work has been partially supported by Ministerio de
Ciencia e Innovaci\'on (Spain) under project MTM2007-61572
(co-financed by the ERDF of the European Union). Additional financial support from the 
Generalitat Valenciana through project GV/2009/032 (SB), Fundaci\'o Bancaixa (FC)
 and Universidad del Pa\'{\i}s
Vasco/Euskal Herriko Uniberstsitatea through project EHU08/43 (AM) is also acknowledged.

\bibliographystyle{plain}
\bibliography{ourbib,spectral,geom_int,numerbib}

\end{document}